\documentclass[10pt,a4paper]{article}

\usepackage[utf8]{inputenc} 
\usepackage[english]{babel} 

\usepackage{graphicx}
\usepackage{mathtools, cuted}
\usepackage{hyperref}

\usepackage{mathtools}
\usepackage{amsmath, amssymb}

\usepackage{setspace}

\usepackage{amssymb}

\usepackage[a4paper, total={6in, 10in}]{geometry}

\usepackage{amsthm}

\newtheorem{theorem}{Theorem}[section]

\newtheorem{corollary}[theorem]{Corollary}
\newtheorem{lemma}[theorem]{Lemma} 
\newtheorem{proposition}[theorem]{Proposition} 
\newtheorem{remark}[theorem]{Remark}

\date{}

\begin{document}

\title{Singularity in CM Sequences}
\author{Reza Rezaie and X. Rong Li
\thanks{The authors are with the Department of Electrical Engineering, University of New Orleans, New Orleans, LA 70148. Email addresses are
 {\tt\small rrezaie@uno.edu} and {\tt\small xli@uno.edu}. Research was supported by NASA through grant NNX13AD29A.}}

\maketitle

\begin{abstract}
Most existing results about modeling and characterizing Gaussian Markov, reciprocal, and conditionally Markov (CM) processes assume nonsingularity of the processes. This assumption makes the analysis easier, but restricts application of these processes. This paper studies, models, and characterizes the general (singular/nonsingular) Gaussian CM (including reciprocal and Markov) sequence. For example, to our knowledge, there is no dynamic model for the general (singular/nonsingular) Gaussian reciprocal sequence in the literature. We obtain two such models from the CM viewpoint. As a result, the significance of studying reciprocal sequences from the CM viewpoint is demonstrated. The results of this paper unify singular and nonsingular Gaussian CM (including reciprocal and Markov) sequences and provide tools for their application.

\end{abstract}

\textbf{Keywords:} Conditionally Markov, reciprocal, Markov, Gaussian, singular, dynamic model, characterization.

\section{Introduction}\label{Intro}

Consider stochastic sequences defined over $[0,N]=( 0,1,\ldots,N)$. For convenience, let the index be time. A sequence is Markov if and only if (iff) conditioned on the state at any time $k$, the segment before $k$ is independent of the segment after $k$. A sequence is reciprocal iff conditioned on the states at any two times $j$ and $l$, the segment inside the interval $(j,l)$ is independent of the two segments outside $[j,l]$. In other words, inside and outside are independent given the boundaries. A sequence is $CM_F$ ($CM_L$) over $[k_1,k_2]$ iff conditioned on the state at time $k_1$ ($k_2$), the sequence is Markov over $[k_1+1,k_2]$ ($[k_1,k_2-1]$) \cite{CM_Part_I_Conf}. The set of CM sequences is very large and it includes many classes. Every Markov sequence is a reciprocal sequence and every reciprocal sequence is a CM sequence.

Markov processes have been widely studied and used for modeling random problems in a lot of applications. However, they are not general enough for some problems (e.g., \cite{Fanas1}--\cite{Picci}), and more general processes are needed. Reciprocal processes are one generalization of Markov processes. CM processes (including reciprocal processes as a special case) provide a systematic approach for generalization of Markov processes (based on conditioning) leading to various classes of processes.

Reciprocal processes have been used in many different areas. In a quantized state space, \cite{Fanas1}--\cite{White_Tracking1} used finite-state reciprocal sequences and their generalization for detection of anomalous trajectory pattern, intent inference, and tracking. \cite{White_Gaussian} studied the generalized reciprocal sequence of \cite{White_Tracking1} in the Gaussian case. The idea of \cite{Simon_1}--\cite{Simon2} for intent inference, e.g., in an intelligent interactive vehicle's display, can be interpreted in a reciprocal process setting. \cite{Picci} used reciprocal processes in image processing. \cite{Levy_2} studied reciprocal processes in the context of stochastic mechanics. Also, in \cite{DD_Conf}--\cite{DW_Conf}, we used some CM sequences for motion trajectory modeling.

An application of CM sequences is in trajectory modeling. Consider the trajectory of a target from an origin to a destination. To emphasize that the trajectory ends up at the destination, we call it a destination-directed trajectory \cite{DD_Conf}. Let it be modeled by a sequence $[x_k]_0^N=(x_0,x_1,\ldots, x_N)$. One can model the main elements of such a trajectory (i.e., an origin, a destination, and motion in between) as follows. The origin, the destination, and their relationship is modeled by a joint density of $x_0$ and $x_N$. Assuming the density of $x_N$ is known, the evolution law can be modeled as a conditional density given the state at destination $x_N$. Different conditional densities correspond to different evolution laws. The simplest conditional density is the one which is equal to the product of its marginals: $p([x_k]_0^{N-1}|x_N)=\prod _{k=0}^{N-1}p(x_k|x_N)$. But this is often inadequate. Then, the next is a Markov conditional density: $p([x_k]_0^{N-1}|x_N)=p(x_0|x_N)\prod _{k=1}^{N-1}p(x_k|x_{k-1},x_N)$. This evolution law corresponds to a $CM_L$ sequence\footnote{By a $CM_L$ sequence we mean a sequence which is $CM_L$ over $[0,N]$.}. The main components of a $CM_L$ sequence $[x_k]_0^N$ are: a joint density of $x_0$ and $x_N$---in other words, an initial density and a final density conditioned on the initial, or equivalently, the other way round---and an evolution law that is conditionally Markov (conditioned on $x_N$). The above argument naturally leads to $CM_L$ sequences for modeling destination-directed trajectories. Following the same argument, we can consider more general and complicated evolution laws, if necessary. So, by choosing conditional densities, different destination-directed trajectories can be modeled.

Gaussian CM processes were introduced in \cite{Mehr} based on mean and covariance functions, where conditioning on the state at the first index (time) of the CM interval was considered. Also, it was assumed that the processes are nonsingular on the interior of the time interval. In \cite{CM_Part_I_Conf}, definitions of different (Gaussian/non-Gaussian) CM processes were presented based on conditioning on the state at the first or the last time of the CM interval, (stationary/non-stationary) nonsingular Gaussian CM sequences were studied, and their dynamic models and characterizations were obtained. \cite{ABRAHAM} commented on the relationship between the Gaussian CM process and the Gaussian reciprocal process. \cite{CM_Part_II_A_Conf} presented the relationship between the CM process and the reciprocal process for the general (Gaussian/non-Gaussian) case. Inspired by \cite{ABRAHAM}, in \cite{CM_Part_II_B_Conf} we obtained a representation of some nonsingular Gaussian CM sequences as a sum of a nonsingular Gaussian Markov sequence and an uncorrelated nonsingular Gaussian vector.

Reciprocal processes were introduced in \cite{Bernstein}, and studied in \cite{Jamison_Reciprocal}--\cite{Moura2}, \cite{ABRAHAM}--\cite{CM_Part_II_B_Conf}, \cite{White}--\cite{White_3}, \cite{Krener}--\cite{Carm}, and others. In \cite{Levy_Dynamic}, which is a significant paper on Gaussian reciprocal sequences, a dynamic model and a characterization of the nonsingular Gaussian reciprocal sequence were presented. Due to the dynamic noise correlation and the nearest-neighbor structure of that model, it is not easy to apply. For example, recursive estimation of a reciprocal sequence based on the model of \cite{Levy_Dynamic} is not straightforward and \cite{Levy_Dynamic}--\cite{Moura2} tried to find such a recursive estimator. A characterization of the nonsingular Gaussian Markov sequence (which is a special reciprocal sequence) was presented in \cite{Ackner}. In \cite{CM_Part_II_A_Conf}--\cite{CM_Part_II_B_Conf} we studied reciprocal sequences from the CM viewpoint and obtained a dynamic model with white dynamic noise for the nonsingular Gaussian reciprocal sequence. Recursive estimation of nonsinguar Gaussian reciprocal sequences based on that model is straightforward.

From the viewpoint of singularity, one can consider two extreme cases for Gaussian sequences. One extreme is a sequence being almost surely constant throughout the time interval. The other extreme is a nonsingular sequence, i.e., a sequence with a nonsingular covariance matrix. A singular sequence is in general between these two extreme cases (including the first case). For example, a Gaussian sequence can be singular because it is almost surely constant at a time (i.e., the state at a time is almost surely constant), or because the states of the sequence at two times are almost surely linearly dependent. There are various such causes (corresponding to different times) leading to singular Gaussian sequences. As a result, we have various singularities. It is desired to model and characterize all singular and nonsingular Gaussian sequences in a unified way.

Characterizations of nonsingular Gaussian Markov, reciprocal, and CM sequences were presented in \cite{Ackner}, \cite{Levy_Dynamic}, \cite{CM_Part_I_Conf}. However, these characterizations, which are based on the inverse of the covariance matrix of the whole sequence, do not work for singular sequences. In \cite{Doob} a characterization was presented for the scalar-valued (singular/nonsingular) Gaussian Markov process in terms of the covariance function. However, that characterization does not work for the general vector-valued case. In \cite{Krener} a characterization was presented for a special kind of nonsingular Gaussian reciprocal processes (i.e., second-order nonsingular Gaussian processes, that is, Gaussian processes with covariance matrices corresponding to any two times of the process being nonsingular). \cite{Chen} presented a characterization of the Gaussian reciprocal process based on the Markov property. That characterization is actually a representation of the reciprocal process in terms of the Markov process and is specifically for continuous-time processes. \cite{Carm} presented a different characterization of the Gaussian reciprocal process based on the Markov property. Characterizations of \cite{Chen} and \cite{Carm} converted the question about a characterization of the Gaussian reciprocal process to the question about a characterization (in terms of the covariance function) of the Gaussian Markov process, which was left unanswered for the general vector-valued Gaussian process. Later studies on the covariance of Gaussian processes were mainly under some nonsingularity assumption \cite{G_1}--\cite{G_3}. Despite the above attempts, to our knowledge, there is no characterization in terms of the covariance function for the general (singular/nonsingular) Gaussian CM (including reciprocal and Markov) process in the literature.    

The well-posedness of the reciprocal dynamic model presented in \cite{Levy_Dynamic} (i.e., the uniqueness of the sequence obeying the model) is guaranteed by the nonsingularity assumption for the covariance of the whole sequence. Dynamic models of different classes of nonsingular Gaussian CM sequences were presented in \cite{CM_Part_I_Conf}. It can be seen that unlike the model of \cite{Levy_Dynamic}, the nonsingularity assumption is not critical for the well-posedness of the CM models of \cite{CM_Part_I_Conf}. Following \cite{CM_Part_I_Conf}, \cite{CM_Part_II_A_Conf} obtained a dynamic model for the nonsingular Gaussian reciprocal sequence form the CM viewpoint. However, that model does not work for singular sequences, although the nonsingularity assumption is not critical for its well-posedness. To our knowledge, there is no dynamic model for the Gaussian reciprocal sequence\footnote{In the rest of the paper, by the ``Gaussian sequence" we mean the general singular/nonsingular Gaussian sequence. Otherwise, we make it explicit if we only mean the nonsingular Gaussian sequence (i.e., covariance of the whole sequence being nonsingular).} in the literature. For example, it is not clear how the model of \cite{Levy_Dynamic} can be extended to the Gaussian reciprocal sequence. More generally, there is no dynamic model for Gaussian CM sequences in the literature.      

Although they make the analysis and modeling easy, nonsingularity assumptions restrict application of Gaussian CM (including reciprocal and Markov) sequences. Without such assumptions, we have a larger and more powerful set of sequences for modeling problems. Some problems can be modeled by a singular sequence better than a nonsingular one. For example, \cite{DD_Conf} used a nonsingular Gaussian $CM_L$ sequence for trajectory modeling between an origin and a destination. Now assume that the origin/destination is known, i.e., some components of the state of the sequence at the origin/destination are almost surely constant. Then, a singular $CM_L$ sequence is a better model for such trajectories. 

The main goal of this paper is threefold: 1) to obtain dynamic models and characterizations of the general Gaussian CM (including reciprocal and Markov) sequence to unify singular and nonsingular Gaussian CM sequences theoretically, 2) to provide tools for application of (singular/nonsingular) Gaussian CM sequences, e.g., in trajectory modeling with destination information, 3) to emphasize the significance of studying reciprocal sequences from the CM viewpoint (which is different from the viewpoint of the literature), e.g., by obtaining dynamic models for the general Gaussian reciprocal sequence from the CM viewpoint. 

The main contributions of this paper are as follows. Dynamic models and characterizations of (singular/nonsingular) Gaussian CM, reciprocal, and Markov sequences are obtained. Two types of characterizations are presented for Gaussian CM and reciprocal sequences. The first type is in terms of the covariance function of the sequence. The second type, which is similar in spirit to (but different from) those of \cite{Chen} and \cite{Carm}, is based on the state concept in system theory (i.e., the Markov property). By deriving a characterization for the general vector-valued Gaussian Markov sequence in terms of the covariance function, we can check the Markov property. Then, the second type of characterization of Gaussian CM and reciprocal sequences (which is based on the Markov property) becomes complete and makes better sense. It is shown that dynamic models of Gaussian CM sequences have a structure similar to those of nonsingular Gaussian CM sequences \cite{CM_Part_I_Conf}, and they differ in the values of their parameters. Therefore, the presented models unify singular and nonsingular Gaussian CM sequences. No dynamic model for general singular/nonsingular Gaussian reciprocal sequences is available. For example, it is not clear how the reciprocal model of \cite{Levy_Dynamic} (which is only for nonsingular Gaussian reciprocal sequences) can be extended to the general singular/nonsingular case even after so many years. We obtain two dynamic models for the Gaussian reciprocal sequence from the CM viewpoint (which is different from the viewpoint of the literature on reciprocal sequences). As a result, the significance and the fruitfulness of studying reciprocal sequences from the CM viewpoint is demonstrated. A full spectrum of models (characterizations) ranging from a $CM_L$ model (characterization) to a reciprocal $CM_L$ model\footnote{Every reciprocal sequence is a $CM_L$ sequence. A $CM_L$ model describing a reciprocal sequence is called a reciprocal $CM_L$ model.} (characterization) is presented for the Gaussian case. The obtained models and characterizations unify singular and nonsingular Gaussian CM sequences. The Markov-based representations of nonsingular Gaussian CM sequences presented in \cite{CM_Part_II_B_Conf} are extended to the general singular/nonsingular case. 
  
In \cite{CM_Part_I_Conf} and \cite{CM_Part_II_A_Conf}--\cite{CM_Part_II_B_Conf} dynamic models and characterizations of nonsingular Gaussian CM (including reciprocal) sequences were presented. They were applied to trajectory modeling in \cite{DD_Conf}--\cite{DW_Conf}. The nonsingularity assumption is required for models and characterizations obtained in \cite{CM_Part_I_Conf}, \cite{CM_Part_II_A_Conf}--\cite{CM_Part_II_B_Conf}. For example, the characterizations are based on the inverse of the covariance matrix of the whole sequence. But the covariance of a singular sequence is not invertible. Also, proofs of models presented in \cite{CM_Part_I_Conf} and \cite{CM_Part_II_A_Conf}--\cite{CM_Part_II_B_Conf} are based on the nonsingularity of sequences and do not work for the singular case. In this paper, we use different ideas and approaches to obtain dynamic models and characterizations for the general singular/nonsingular case. 

The paper is organized as follows. Section \ref{Section_Definition} reviews some definitions and preliminaries required in the next sections. Dynamic models and characterizations of Gaussian $CM_L$ ($CM_F$) and Gaussian reciprocal sequences are obtained in Section \ref{Section_CM} and Section \ref{Section_Reciprocal}, respectively. Models and characterizations of some other CM classes are presented in Section \ref{Section_Other_CM}. Section \ref{Section_Summary} contains a summary and conclusions.

\section{Definitions and Preliminaries}\label{Section_Definition}

We consider sequences defined over $[0,N]$, which is a general discrete index interval, but for simplicity it is called time. The following conventions are used for sequences and time intervals:
\begin{align*}
[i,j]& \triangleq ( i,i+1,\ldots ,j-1,j ), \quad i<j, \quad i,j \in [0,N] \\
[x_k]_{i}^{j} & \triangleq ( x_i, x_{i+1}, \ldots, x_j ), \quad [x_k] \triangleq [x_k]_{0}^{N}
\end{align*}
where $k$ in $[x_k]_i^j$ is a dummy variable. We consider $k_1,k_2, l_1, l_2 \in [0,N]$. The symbols `` $ ' $ " and ``$\setminus $" are used for transposition and set subtraction, respectively. In addition, $0$ may denote a zero scalar, vector, or matrix, as is clear from the context. For a matrix $P$, $P_{i,j}$ denotes the (block) entry at (block) row $i+1$ and (block) column $j+1$ of $P$. Also, $P_i \triangleq P_{i,i}$. For example, $C$ is the covariance matrix of the whole sequence $[x_k]$, $C_{i,j}$ is the covariance function\footnote{$i,j \in [0,N]$, but matrix $C$ has (block) rows (columns) $1$ to $N+1$.}, and $C_i \triangleq C_{i,i}$. $F(\cdot | \cdot)$ denotes the conditional cumulative distribution function (CDF). $E[\cdot]$ denotes the expectation operator. The abbreviation ZMG is used for ``zero-mean Gaussian". Some equations and statements hold almost surely (and not strictly), as is clear from the context. For clarity, in some cases we mention it explicitly. The abbreviation ``a.s." stands for ``almost surely".

Formal measure-theoretical definitions of CM and reciprocal processes can be found in \cite{CM_Part_I_Conf}, \cite{Jamison_Reciprocal}. Here, we present definitions in a simple language.

A sequence $[x_k]$ is $[k_1,k_2]$-$CM_c, c \in \lbrace k_1,k_2 \rbrace$, (i.e., CM over $[k_1,k_2]$) iff conditioned on the state at time $c=k_1$ ($c=k_2$), the sequence is Markov over $[k_1+1,k_2]$ ($[k_1,k_2-1]$). The above definition is equivalent to the following lemma \cite{CM_Part_I_Conf}.

\begin{lemma}\label{CMc_CDF}
$[x_k]$ is $[k_1,k_2]$-$CM_c, c \in \lbrace k_1,k_2 \rbrace$, iff $ F(\xi _k|[x_{i}]_{k_1}^{j},$ $x_{c}) =F(\xi _k|x_j,x_c)$ \footnote{$F(\xi_k|x_j)=P\lbrace x^1_k\leq \xi^1_k, x^2_k\leq \xi^2_k, \ldots , x^d_k\leq \xi^d_k|x_j \rbrace$, where for example $x^1_k$ and $\xi^1_k$ are the first entries of the vectors $x_k$ and $\xi_k$, respectively. Similarly for other CDFs.}, $\forall j,k \in [k_1,k_2], j<k$, $\forall \xi _k \in \mathbb{R}^d$, where $d$ is the dimension of $x_k$.  

\end{lemma}

The interval $[k_1,k_2]$ of a $[k_1,k_2]$-$CM_c$ sequence is called the \textit{CM interval} of the sequence.

\begin{remark}\label{R_CMN}
We use the following notation ($ k_1<k_2$)
\begin{align*}
[k_1,k_2]\text{-}CM_c=\left\{ \begin{array}{cc} 
[k_1,k_2]\text{-}CM_F & \text{if  } c=k_1\\ \relax
[k_1,k_2]\text{-}CM_L & \text{if  } c=k_2
\end{array} \right.
\end{align*}
where the subscript ``$F$" or ``$L$" is used because the conditioning is at the \textit{first} or the \textit{last} time of the CM interval. 

\end{remark}

\begin{remark}\label{R_CMN_2}
The $[0,N]$-$CM_c$ sequence is called $CM_c$, $c \in \lbrace 0,N \rbrace$; i.e., the CM interval is dropped if it is the whole time interval.

\end{remark}

A $CM_0$ ($CM_N$) sequence is called $CM_F$ ($CM_L$). Let $[x_k]$ be a $[k_1,k_2]$-$CM_c$ sequence. By Remark \ref{R_CMN_2}, $[x_k]_{k_1}^{k_2}$ is a $CM_c$ sequence. By a $CM_L \cap [k_1,N]$-$CM_F$ sequence we mean a sequence being both $CM_L$ and $[k_1,N]$-$CM_F$. We define that every sequence with a length smaller than 3 (i.e., $\lbrace x_0,x_1 \rbrace$, $\lbrace x_0 \rbrace$, and $\lbrace  \rbrace$) is Markov. Similarly, every sequence is $[k_1,k_2]$-$CM_c$, $|k_2 - k_1| <3$. So, $CM_L$ and $CM_L \cap [k_1,N]$-$CM_F$, $k_1 \in [N-2,N]$ are equivalent.

We have the following lemma \cite{CM_Part_II_A_Conf}, \cite{CM_Part_I_Conf}, \cite{Jamison_Reciprocal}.
\begin{lemma}\label{CDF}
$[x_k]$ is reciprocal iff $F(\xi _k|[x_{i}]_{0}^{j},[x_i]_l^N)=F(\xi _k|$ $x_j,x_l), \forall j,k,l \in [0,N], j < k < l$, $\forall \xi _k \in \mathbb{R}^d$, where $d$ is the dimension of $x_k$. 

\end{lemma}

For the Gaussian case, we have the following lemmas for CM and Markov sequences \cite{CM_Part_I_Conf}. 

\begin{lemma}\label{GaussianCMc_Definition_Expectation}
A Gaussian $[x_k]$ is $[k_1,k_2]$-$CM_c, c \in \lbrace k_1,k_2 \rbrace$, iff  $E[x_k|[x_{i}]_{k_1}^{j},x_{c}]=E[x_k|x_j,x_{c}]$, $\forall j,k \in [k_1,k_2], j<k$. 

\end{lemma}

\begin{lemma}\label{GaussianMarkov_Definition_E}
A Gaussian $[x_k]$ is Markov iff  $E[x_k|[x_i]_0^j]=E[x_k|x_j]$, $\forall j, k \in [0,N], j < k$.  

\end{lemma}

By definitions, every reciprocal sequence is $[0,k_2]$-$CM_L$ and $[k_1,N]$-$CM_F$, $\forall k_1,k_2 \in [0,N]$. The following theorem gives the relationship between CM and reciprocal sequences \cite{CM_Part_II_A_Conf}.

\begin{theorem}\label{CM_iff_Reciprocal}
$[x_k]$ is reciprocal iff it is (i) $CM_L$ and $[k_1,N]$-$CM_F,$ $ \forall k_1 \in [0,N]$, or equivalently (ii) $CM_F$ and $[0,k_2]$-$CM_L,$ $\forall k_2 \in [0,N]$.

\end{theorem}

\section{Dynamic Model and Characterization of $CM_c$ Sequences}\label{Section_CM}

\subsection{Dynamic Model}

The following theorem presents a model of ZMG $CM_c$ sequences called a $CM_c$ model. A Gaussian sequence is $CM_c$ iff its zero-mean part is $CM_c$. So, based on Theorem \ref{CMc_Dynamic_Forward_Theorem}, a model of nonzero-mean Gaussian $CM_c$ sequences is obtained.

\begin{theorem}\label{CMc_Dynamic_Forward_Theorem}
A ZMG $[x_k]$ is $CM_c, c \in \lbrace 0,N \rbrace$, iff it obeys
\begin{align}
x_k=G_{k,k-1}x_{k-1}+G_{k,c}x_c+e_k, \quad k \in [1,N] \setminus \lbrace c \rbrace
\label{CMc_Dynamic_Forward}
\end{align}
where $[e_k]$ is a zero-mean white Gaussian sequence with $G_k=\text{Cov}(e_k)$, and boundary condition\footnote{Note that $\eqref{CMc_Forward_BC2}$  means that for $c=N$ we have $x_N=e_N$ and $x_0=G_{0,N}x_N+e_0$; for $c=0$ we have $x_0=e_0$. Likewise for $\eqref{CMc_Forward_BC1}$.}
\begin{align}
&x_c=e_c, \quad x_0=G_{0,c}x_c+e_0 \, \, (\text{for} \, \, c=N) \label{CMc_Forward_BC2}
\end{align}
or equivalently\footnote{$e_0$ and $e_N$ in $\eqref{CMc_Forward_BC2}$ are not necessarily the same as $e_0$ and $e_N$ in $\eqref{CMc_Forward_BC1}$. Just for simplicity we use the same notation.}
\begin{align}
&x_0=e_0, \quad x_c=G_{c,0}x_0+e_c \, \, (\text{for} \,\, c=N)\label{CMc_Forward_BC1}
\end{align}    

\end{theorem}
\begin{proof}
Necessity: We first prove it for $c=N$ (i.e., $CM_L$). Let $[x_k]$ be a ZMG $CM_L$ sequence with covariance function $C_{l_1,l_2}$. It is shown that $[x_k]$ is modeled by $\eqref{CMc_Dynamic_Forward}$ along with $\eqref{CMc_Forward_BC2}$ or $\eqref{CMc_Forward_BC1}$. First, we obtain boundary condition $\eqref{CMc_Forward_BC1}$. Let $x_0=e_0$, where $e_0$, a ZMG vector with covariance $C_0$, is defined for notational unification. The conditional expectation $E[x_N|x_0]$ is the a.s. unique Borel measurable function of $x_0$ for which\footnote{Note that there is no density function for singular Gaussian sequences.}
\begin{align}
E[(x_N-E[x_N|x_0])g(x_0)]=0\label{Conditional_Expectation_BC1}
\end{align}
for every Borel measurable function $g$.

We now show the existence of $B$ for which $E[(x_N-Bx_0])g(x_0)]=0$ for every Borel measurable function $g$. Then, by the uniqueness of the conditional expectation in $\eqref{Conditional_Expectation_BC1}$, we conclude $E[x_N|x_0]=Bx_0$ \cite{Doob}, \cite{Anan} and obtain $B$. 

The following equation
\begin{align}
BC_{0}=C_{N,0}\label{Normal_Eq_BC1}
\end{align}
has a solution $B=C_{N,0}C_{0}^++S(I-C_{0}C_{0}^+)$ for any matrix $S$, where the superscript ``$+$" means the Moore Penrose inverse (MP-inverse) \cite{Li}. We have $BC_0=(C_{N,0}C_{0}^++S(I-C_{0}C_{0}^+))C_0=C_{N,0}C^+_0C_0$. So, to show that $B=C_{N,0}C_{0}^++S(I-C_{0}C_{0}^+)$ is a solution of $\eqref{Normal_Eq_BC1}$, it suffices to show that $C_{N,0}C^+_0C_0=C_{N,0}$. To show it, define $u \triangleq (I-C_0C^+_0)x_0$. Clearly, $E[u]=(I-C_0C^+_0)E[x_0]=0$ and $\text{Cov}(u)=(I-C_0C_0^+)C_0(I-C_0C_0^+)'=0$. So, $u = (I-C_0C^+_0)x_0=0$ or $x_0=C_0C^+_0x_0$ (a.s.). Then, $C_{N,0}=E[x_Nx_0']=E[x_N(C_0C^+_0x_0)']=C_{N,0}C^+_0C_0$, where $(C^+_0)'=C^+_0$ (see \cite{Matrix_CB} for properties of the MP-inverse). 

Since $[x_k]$ is zero-mean, $\eqref{Normal_Eq_BC1}$ can be rewritten as 
\begin{align}
E[(x_N-Bx_0)x_0']=0\label{Cond_Gaussian_BC1}
\end{align}
which means $x_N-Bx_0$ is uncorrelated with (and orthogonal to, because $[x_k]$ is zero-mean) $x_0$. Due to the Gaussianity of $[x_k]$, $x_N-Bx_0$ and $x_0$ are independent and we have
\begin{align}
E[(x_N-Bx_0)g(x_0)]=0\label{Cond_Gaussian}
\end{align}
for every Borel measurable function $g$. Comparing $\eqref{Conditional_Expectation_BC1}$ and $\eqref{Cond_Gaussian}$, and by the uniqueness of the conditional expectation, we have $E[x_N|x_0]=Bx_0$ for $B$ given above (i.e., solution of $\eqref{Normal_Eq_BC1}$). Also, $E[x_N|x_0]=C_{N,0}C_0^{+}x_0$ since $(I-C_{0}C_{0}^+)x_0\stackrel{a.s.}{=}0$. We define $e_N$ as $e_N =x_N-C_{N,0}C_{0}^+x_0$. By $\eqref{Cond_Gaussian_BC1}$, $e_N$ and $e_0$ are uncorrelated. Also, the covariance of $e_N$ is $C_N-C_{N,0}C_0^+C_{N,0}'$. 

We can obtain $\eqref{CMc_Forward_BC2}$ as $x_N=e_N$ and $x_0=C_{0,N}C_{N}^+x_N+ e_0$, 
where $e_N$ and $e_0$ are uncorrelated ZMG vectors with covariances $C_N$ and $C_0-C_{0,N}C_N^+C_{0,N}'$, respectively.

Following a similar argument as above, based on the definition of the conditional expectation $E[x_k|y_{k-1}]$, $y_k=[x_k',x_N']'$, we obtain $E[x_k|y_{k-1}]=A_ky_{k-1}$, where $A_k=C^{xy}_{k,k-1}(C^y_{k-1})^++S(I-C^y_{k-1}(C^y_{k-1})^+)$, $C^y_{k-1}=\text{Cov}(y_{k-1})$, and $C^{xy}_{k,k-1}=\text{Cov}(x_k,y_{k-1})$. In addition, we have $(I-C^{y}_{k-1}(C^y_{k-1})^+)y_{k-1}\stackrel{a.s.}{=}0$, because $\text{Cov}((I-C^{y}_{k-1}(C^y_{k-1})^+)y_{k-1})=0$ and $E[(I-C^{y}_{k-1}(C^y_{k-1})^+)y_{k-1}]=0$. Thus, we have a.s.
\begin{align}
E[x_k|y_{k-1}]=C^{xy}_{k,k-1}(C^y_{k-1})^+y_{k-1}\label{Gaussian_Linear}
\end{align}
We define $e_k$, $\forall k \in [1,N-1]$, as
\begin{align}
e_k=x_k-E[x_k|x_{k-1},x_N]\label{CML_Dynamic_Covariance}
\end{align}
where $[e_k]$ is a zero-mean white Gaussian sequence (with covariances $G_k=C_k - C^{xy}_{k,k-1}(C^y_{k-1})^+(C^{xy}_{k,k-1})'$, $k \in [1,N-1]$), which can be verified as follows. By the definition of the conditional expectation $E[x_k|[x_i]_0^{k-1},x_N]$, we have
\begin{align}
E[(x_k-E[x_k|[x_i]_0^{k-1},x_N])g([x_i]_0^{k-1},x_N)]=0\label{Cond_E_General}
\end{align}
for every Borel measurable function $g$. Then, by Lemma \ref{GaussianCMc_Definition_Expectation}, $\eqref{Cond_E_General}$ leads to
\begin{align}
E[(x_k-E[x_k|x_{k-1},x_N])g([x_i]_0^{k-1},x_N)]=0\label{Cond_E_CML}
\end{align}
Since $x_k-E[x_k|x_{k-1},x_N]$ is uncorrelated with $g([x_i]_0^{k-1},$ $x_N)$, it can be seen from $\eqref{CML_Dynamic_Covariance}$ that $[e_k]$ is white ($E[e_ke_{j}']=0$, $k\neq j$). Thus, given any ZMG $CM_L$ sequence, its evolution obeys $\eqref{CMc_Dynamic_Forward}$ along with $\eqref{CMc_Forward_BC2}$ or $\eqref{CMc_Forward_BC1}$. 

Proof of necessity for $c=0$ (i.e., $CM_F$) is similar. We have $x_0=e_0$, $x_1=C_{1,0}C_0^{+}x_0+e_1$, and $x_k=C^{xy}_{k,k-1}(C^y_{k-1})^{+}y_{k-1}+e_k, k \in [2,N]$, where $G_0=C_0$, $G_1=C_1-C_{1,0}C_0^{+}C_{1,0}'$, and $G_k=C_k - C^{xy}_{k,k-1}(C^y_{k-1})^{+}(C^{xy}_{k,k-1})', k \in [2,N]$.

Sufficiency: Our proof of sufficiency is similar to that of the zero-mean nonsingular Gaussian $CM_c$ model \cite{CM_Part_I_Conf}. From $\eqref{CMc_Dynamic_Forward}$, we have $x_k=G_{k,j}x_j+G_{k,c|j}x_c+e_{k|j}$, where $G_{k,j}$ and $G_{k,c|j}$ can be obtained from parameters of $\eqref{CMc_Dynamic_Forward}$, and $e_{k|j}$ is a linear combination of $[e_l]_{j+1}^k$. Since $[e_k]$ is white, $[e_l]_{j+1}^k$ (and so $e_{k|j}$) is uncorrelated with $[x_k]_0^j$ and $x_c$. So, we have $E[x_k|[x_{i}]_{0}^{j},x_c]=E[x_k|x_j,x_c]$. Then, by Lemma \ref{GaussianCMc_Definition_Expectation}, $[x_k]$ is $CM_c$. 

$\eqref{L_2}$ and $\eqref{F}$ (below) are always nonsingular. Then, by $\eqref{Mxe}$, $\eqref{CMc_Dynamic_Forward}$--$\eqref{CMc_Forward_BC2}$ (for every parameter value) admit a unique covariance function (i.e., a unique sequence). Similarly, $\eqref{CMc_Dynamic_Forward}$ and $\eqref{CMc_Forward_BC1}$ admit a unique covariance function for every parameter value.
\end{proof}

The boundary conditions $\eqref{CMc_Forward_BC2}$ and $\eqref{CMc_Forward_BC1}$ are equivalent. So, later we only consider one of them.
 
Consider $\eqref{CMc_Dynamic_Forward}$--$\eqref{CMc_Forward_BC2}$ for $c=N$. We have
\begin{align}
\mathcal{G}x&=e \label{Mxe}
\end{align}
where $e \triangleq [e_0' , e_1', \ldots  , e_N']'$, $x \triangleq [x_0' , x_1' ,  \ldots , x_N']'$, and $\mathcal{G}$ is 
\begin{align}\label{L_2}
\left[ \begin{array}{cccccc}
I & 0 & 0 &  \cdots & 0 & -G_{0,N}\\
-G_{1,0} & I & 0 &  \cdots & 0 & -G_{1,N}\\
0 & -G_{2,0} & I & 0 & \cdots & -G_{2,N}\\
\vdots & \vdots & \vdots & \vdots & \vdots & \vdots \\
0 & 0 & \cdots & -G_{N-1,N-2} & I & -G_{N-1,N}\\
0 & 0 & 0 &  \cdots & 0 & I
\end{array}\right]
\end{align}
From $\eqref{Mxe}$, the covariance matrix of $x$ (i.e., $C$) is calculated as 
\begin{align}
C=\mathcal{G}^{-1}G(\mathcal{G}')^{-1}\label{C_Inverse}
\end{align}
where $G=\text{diag}(G_0,\ldots,G_N)$. Similarly, for $c=0$, the covariance is given by $\eqref{C_Inverse}$, where $G=\text{diag}(G_0,\ldots,G_N)$ and $\mathcal{G}$ is
\begin{align}\label{F}
\left[ \begin{array}{cccccc}
I & 0 & 0 &  \cdots & 0 & 0\\
-2G_{1,0} & I & 0 &  \cdots & 0 & 0\\
-G_{2,0} & -G_{2,1} & I & 0 & \cdots & 0\\
\vdots & \vdots & \vdots & \vdots & \vdots & \vdots \\
-G_{N-1,0} & 0 & \cdots & -G_{N-1,N-2} & I & 0\\
-G_{N,0} & 0 & 0 &  \cdots & -G_{N,N-1} & I
\end{array}\right]
\end{align}

By $\eqref{C_Inverse}$, we can determine the imposed condition on the parameters of $\eqref{CMc_Dynamic_Forward}$--$\eqref{CMc_Forward_BC2}$ due to a specific singularity. An example follows.

\begin{corollary}\label{CML_Dynamic_Forward_Corollary_2nd}
A ZMG $[x_k]$ with covariance function $C_{l_1,l_2}$ is $CM_L$ with the matrices
\begin{align}
\left[\begin{array}{cc}
C_k & C_{k,N} \\
C_{N,k} & C_{N}
\end{array}\right], k \in [0,N-2] \label{TC_CML}
\end{align}
being nonsingular iff 
\begin{align}
x_k&=G_{k,k-1}x_{k-1} + G_{k,N}x_N+e_k, k \in [1,N-1]\label{CML_Dynamic_Forward_2nd}\\
x_N&=e_N, \quad x_0=G_{0,N}x_N+e_0\label{CML_Forward_BC2_2nd}
\end{align}
where $[e_k]$ is a zero-mean white Gaussian sequence with $G_k=\text{Cov}(e_k)$, and the matrices 
\begin{align}
\left[\begin{array}{cc}
P_{k} & P_{k,N} \\
P_{N,k} & P_{N}
\end{array}\right], k \in [0,N-2]\label{TP_CML1}
\end{align}
are nonsingular (positive definite\footnote{$P$ is always positive (semi)definite.}), with $P = \mathcal{G}^{-1}G(\mathcal{G}')^{-1}$, $G=\text{diag}(G_0,\ldots,G_N)$, and $\mathcal{G}$ being given by $\eqref{L_2}$.

\end{corollary}
\begin{proof}
A ZMG $[x_k]$ is $CM_L$ iff we have $\eqref{CML_Dynamic_Forward_2nd}$--$\eqref{CML_Forward_BC2_2nd}$. Also, $P$ is the covariance of $[x_k]$ (see $\eqref{C_Inverse}$). So, $\eqref{TC_CML}$ and $\eqref{TP_CML1}$ are equal. 
\end{proof}

By having different values of the parameters, $\eqref{CMc_Dynamic_Forward}$--$\eqref{CMc_Forward_BC2}$ can model all Gaussian $CM_c$ sequences ranging from a nonsingular $CM_c$ sequence to a singular $CM_c$ sequence a.s. zero throughout the time interval. For example, let $|G_k|=0, \forall k \in [0,N]$ ($|\cdot|$ denotes the determinant operator), and all other parameters of $\eqref{CMc_Dynamic_Forward}$--$\eqref{CMc_Forward_BC2}$ be zero. By $\eqref{C_Inverse}$, such a $CM_c$ model is for a white sequence with $|C_k|=0, \forall k \in [0,N]$ (for a scalar-valued sequence, it is actually an a.s. zero sequence). Another extreme is when all the matrices $G_k$ are nonsingular leading to a nonsingular Gaussian $CM_c$ sequence.  

Let $[x_k]$ be a ZMG $CM_L$ sequence. $x_{n}$ and $y_{n-1}=[x_{n-1}' , x_N']'$ are a.s. linearly dependent iff $e_{n}$ is a.s. zero (i.e., $\text{Cov}(e_{n})=0$). It can be verified by $\eqref{CML_Dynamic_Covariance}$. 

Let $[x_k]$ be a ZMG $CM_L$ sequence. $x_{n}$ is a.s. zero iff both $e_{n}$ and $C^{xy}_{n,n-1}(C^y_{n-1})^+y_{n-1}$ are a.s. zero. It is verified as follows. By $\eqref{CML_Dynamic_Covariance}$, $x_{n}$ is a.s. zero iff we have a.s.
\begin{align}
e_{n}+C^{xy}_{n,n-1}(C^y_{n-1})^+y_{n-1}=0\label{S1}
\end{align}
Post-multiplying both sides of $\eqref{S1}$ by $e_{n}'$ and taking expectation, it is concluded that $\text{Cov}(e_{n})=0$, where the fact that $e_{n}$ is orthogonal to $x_{n-1}$ and $x_N$, has been used (see $\eqref{Cond_E_CML}$). Then, by $\eqref{S1}$, we have a.s. $C^{xy}_{n,n-1}(C^y_{n-1})^+y_{n-1}=0$. Therefore, $x_{n}$ is a.s. zero iff both terms of $\eqref{S1}$ are a.s. zero.

\subsection{Characterization}

Two characterizations are presented for Gaussian $CM_c$ sequences with any kind of singularity. The first characterization is as follows.

\begin{theorem}\label{CMc_Characterization_Theorem}
A Gaussian $[x_k]$ with covariance function $C_{l_1,l_2}$ is $CM_c, c \in \lbrace 0,N \rbrace$, iff
\begin{align}
C_{k,i}=\left[ \begin{array}{cc}
C_{k,j} & C_{k,c}
\end{array}\right] \left[\begin{array}{cc}
C_{j} & C_{j,c}\\
C_{c,j} & C_{c}
\end{array}\right]^+ \left[\begin{array}{c}
C_{j,i}\\
C_{c,i}
\end{array}\right]\label{CMc_C}
\end{align}
$\forall i, j, k \in [0,N] \setminus \lbrace c \rbrace$, $ i < j < k$, where the superscript ``$+$" means the MP-inverse. 

\end{theorem}
\begin{proof}
A Gaussian sequence is $CM_c$ iff its zero-mean part is $CM_c$. Also, a sequence and its zero-mean part have the same covariance function. So, it suffices to consider zero-mean sequences. 

Necessity: Let $[x_k]$ be a ZMG $CM_c$ sequence with covariance function $C_{l_1,l_2}$. Define
\begin{align}
r(k,j)=x_k-E[x_k|y_{j}]\label{T1}
\end{align}
$\forall j, k \in [0,N] \setminus \lbrace c \rbrace$, $ j< k$, and $y_{j} \triangleq [ x_j' , x_c' ]'$. Then, since $[x_k]$ is Gaussian, $\eqref{T1}$ leads to (see $\eqref{Gaussian_Linear}$)
\begin{align}
r(k,j)=x_k-\left[ \begin{array}{cc}
C_{k,j} & C_{k,c}
\end{array}\right]\left[ \begin{array}{cc}
C_{j} & C_{j,c}\\
C_{c,j} & C_{c}
\end{array}\right]^{+}\left[ \begin{array}{c}
x_j\\
x_c
\end{array}\right]\label{T2}
\end{align}

On the other hand, by the definition of the conditional expectation $E[x_k|[x_i]_0^{j},x_c]$, we have
\begin{align}
E[(x_k-E[x_k|[x_i]_0^{j},x_c])g([x_i]_0^{j},x_c)]=0\label{T12}
\end{align}
for every Borel measurable function $g$. Then, by Lemma \ref{GaussianCMc_Definition_Expectation}, we have
\begin{align}
E[(x_k-E[x_k|x_{j},x_c])g([x_i]_0^{j},x_c)]=0\label{T13}
\end{align}
By $\eqref{T13}$, $r(k,j)$ is uncorrelated with $[x_i]_0^j$ and $x_c$. So, post-multiplying both sides of $\eqref{T2}$ by $x_i'$, $\forall i \in [0,j-1] \setminus \lbrace c \rbrace$, and taking expectation, we obtain $\eqref{CMc_C}$, where $i,j,k \in [0,N] \setminus \lbrace c \rbrace$, $i < j < k$.

Sufficiency: Let $[x_k]$ be a ZMG sequence with covariance function $C_{l_1,l_2}$ satisfying $\eqref{CMc_C}$, $\forall i,j,k \in [0,N] \setminus \lbrace c \rbrace$, $i<j<k$. Since $[x_k]$ is Gaussian, we have
\begin{align}
E[x_k|x_j,x_c]=\left[ \begin{array}{cc}
C_{k,j} & C_{k,c}
\end{array}\right]\left[ \begin{array}{cc}
C_{j} & C_{j,c}\\
C_{c,j} & C_{c}
\end{array}\right]^{+}\left[ \begin{array}{c}
x_j\\
x_c
\end{array}\right]\label{T9}
\end{align}
 
Define 
\begin{align}
r(k,j)=x_k-\left[ \begin{array}{cc}
C_{k,j} & C_{k,c}
\end{array}\right]\left[ \begin{array}{cc}
C_{j} & C_{j,c}\\
C_{c,j} & C_{c}
\end{array}\right]^{+}\left[ \begin{array}{c}
x_j\\
x_c
\end{array}\right]\label{T4}
\end{align}
where based on $\eqref{CMc_C}$, it is concluded that $r(k,j)$ is uncorrelated with (and since $[x_k]$ is zero-mean, orthogonal to) $[x_i]_0^{j-1} \setminus \lbrace c \rbrace$ (it is seen by post-multiplying both sides of $\eqref{T4}$ by $x_i'$, $\forall i \in [0,j-1] \setminus \lbrace c \rbrace$ and taking expectation). In addition, $r(k,j)$ is orthogonal to $x_j$ and $x_c$. It can be verified based on $\eqref{T9}$ and the definition of the conditional expectation $E[x_k|x_j,x_c]$, where $E[(x_k-E[x_k|x_j,x_c])g(x_j,x_c)]=0$ for every Borel measurable function $g$. Then, due to the Gaussianity, $r(k,j)$ is independent of $[x_i]_0^{j}$ and $x_c$, and consequently $r(k,j)$ is uncorrelated with $g([x_i]_0^{j},x_c)$ for every Borel measurable function $g$. Thus, by the a.s. uniqueness of the conditional expectation in $\eqref{T12}$, 
\begin{align}
E[x_k|&[x_i]_0^j,x_c]=\left[ \begin{array}{cc}
C_{k,j} & C_{k,c}
\end{array}\right]\left[ \begin{array}{cc}
C_{j} & C_{j,c}\\
C_{c,j} & C_{c}
\end{array}\right]^{+}\left[ \begin{array}{c}
x_j\\
x_c
\end{array}\right]\label{Tc1}
\end{align}

So, by $\eqref{T9}$ and $\eqref{Tc1}$, $\forall j,k \in [0,N] \setminus \lbrace c \rbrace$, $j<k$, we have $E[x_k|[x_i]_0^j,x_c]=E[x_k|x_j,x_c]$. Then, by Lemma \ref{GaussianCMc_Definition_Expectation}, $[x_k]$ is $CM_c$.
\end{proof}

The following characterization of the Gaussian $CM_c$ sequence is based on the concept of state in system theory (i.e., Markov property).

\begin{corollary}\label{CMc_Characterization_Markov}
A Gaussian $[x_k]$ is $CM_c$ iff $[y_k] \setminus \lbrace y_c \rbrace $ \footnote{For $c=N$, $[y_k] \setminus \lbrace y_c \rbrace  \triangleq [y_k]_0^{N-1}$, and for $c=0$, $[y_k] \setminus \lbrace y_c \rbrace \triangleq [y_k]_1^N$.} is Markov, where $y_k \triangleq [x_k' , x_c']', \forall k \in [0,N] \setminus \lbrace c \rbrace$. 

\end{corollary}
\begin{proof}
It can be verified by Lemma \ref{GaussianCMc_Definition_Expectation} or Theorem \ref{CMc_Dynamic_Forward_Theorem}.
\end{proof}

\section{Characterization and Dynamic Model of Reciprocal Sequences}\label{Section_Reciprocal}

\subsection{Characterization}

In \cite{Krener} a characterization was presented for the Gaussian reciprocal process $[x_k]$ that is second-order nonsingular, that is, $[x_k]$ for which the covariance of $y=[x_m',x_n']'$ for every $n,m \in [0,N]$ is nonsingular. Inspired by this result, in Theorem \ref{Reciprocal_Characterization_Theorem} below, a characterization of the Gaussian reciprocal sequence is presented. First, we need a corollary of Theorem \ref{CMc_Characterization_Theorem}. By definition, $[x_k]$ is $[k_1,k_2]$-$CM_c$ iff $[x_k]_{k_1}^{k_2}$ is $CM_c$. So, we have the following corollary.
 
\begin{corollary}\label{k1k2CMc_Characterization_Corollary}
A Gaussian $[x_k]$ with covariance function $C_{l_1,l_2}$ is $[k_1,k_2]$-$CM_c, c \in \lbrace k_1,k_2 \rbrace$, iff 
\begin{align}
C_{k,i}=\left[ \begin{array}{cc}
C_{k,j} & C_{k,c}
\end{array}\right] \left[ \begin{array}{cc}
C_{j} & C_{j,c}\\
C_{c,j} & C_{c}
\end{array}\right]^{+} \left[ \begin{array}{c}
C_{j,i}\\
C_{c,i}
\end{array}\right]\label{k1k2_CML}
\end{align} 
$\forall i, j, k \in [k_1,k_2] \setminus \lbrace c \rbrace$, $ i < j < k$.

\end{corollary}

\begin{theorem}\label{Reciprocal_Characterization_Theorem}
A Gaussian $[x_k]$ with covariance function $C_{l_1,l_2}$ is reciprocal iff
\begin{align}
C_{k,i}=\left[ \begin{array}{cc}
C_{k,j} & C_{k,l}
\end{array}\right] \left[ \begin{array}{cc}
C_{j} & C_{j,l}\\
C_{l,j} & C_{l}
\end{array}\right]^{+} \left[ \begin{array}{c}
C_{j,i}\\
C_{l,i}
\end{array}\right]\label{Reciprocal_1}
\end{align} 
(a) $\forall i,j,k,l \in [0,N]$ with $l < i < j < k$, and (b) $\forall i,j,k \in [0,N-1]$ with $i < j < k<l=N$ (or equivalently (a) $\forall i,j,k,l \in [0,N]$ with $i < j < k < l$, and (b) $\forall i,j,k $ $ \in [1,N]$ with $0=l<i < j < k$).

\end{theorem}
\begin{proof}
It follows from Theorem \ref{CM_iff_Reciprocal} and Corollary \ref{k1k2CMc_Characterization_Corollary}.
\end{proof}

First, the characterization presented in \cite{Krener} only works for second-order nonsingular Gaussian reciprocal sequences. The characterization of Theorem \ref{Reciprocal_Characterization_Theorem} works for all Gaussian reciprocal sequences. Second, Theorem \ref{CM_iff_Reciprocal} implies the equality of two sets of sequences, i.e., $\cap _{k_1=0}^N [k_1,N]$-$CM_F \cap _{k_2=0}^N [0,k_2]$-$CM_L = \cap _{k_1=0}^N [k_1,N]$-$CM_F \cap CM_L$. Accordingly, and by Corollary \ref{k1k2CMc_Characterization_Corollary}, for a Gaussian sequence, $\eqref{Reciprocal_1}$ holds for (a) $\forall i,j,k,l \in [0,N]$ with $l < i < j < k $, and (b) $\forall i,j,k,l \in [0,N]$ with $i  < j < k <l $ iff $\eqref{Reciprocal_1}$ holds for (a) $\forall i,j,k,l \in [0,N]$ with $l < i < j < k $, and (b) $\forall i,j,k \in [0,N-1]$ with $i < j < k < l=N$. Although the two conditions are equivalent, the latter is simpler (and more revealing) than the former. It seems \cite{Krener} was not aware of the simpler condition. We obtained the simpler condition based on studying reciprocal sequences from the CM viewpoint, which is different from that of \cite{Krener}. It shows how insightful the CM viewpoint is for studying reciprocal sequences.

Another characterization of the Gaussian reciprocal sequence is based on the concept of state in system theory (i.e. Markov property).

\begin{corollary}\label{Reciprocal_Charactetization_Markov_Proposition}
i) A Gaussian $[x_k]$ is reciprocal iff $[y_k]_{k_1+1}^N$ with $y_k \triangleq [x_k',x_{k_1}']', \forall k \in [k_1+1,N]$, $\forall k_1 \in [0,N]$, and $[y_k]_{0}^{N-1}$ with $y_k \triangleq [x_k',x_N']', \forall k \in [0,N-1]$, are Markov. ii) A Gaussian $[x_k]$ is reciprocal iff $[y_k]_0^{k_2-1}$ with $y_k \triangleq [x_k',x_{k_2}']', \forall k \in [0,k_2-1]$, $\forall k_2 \in [0,N]$, and $[y_k]_{1}^N$ with $y_k \triangleq [x_k',x_0']', \forall k \in [1,N]$, are Markov.

\end{corollary}
\begin{proof}
It follows from Theorem \ref{CM_iff_Reciprocal}, Corollary \ref{CMc_Characterization_Markov}, and the fact that $[x_k]$ is $[k_1,k_2]$-$CM_c$ iff $[x_k]_{k_1}^{k_2}$ is $CM_c$. 
\end{proof}

\subsection{Dynamic Model}

The reciprocal dynamic model of \cite{Levy_Dynamic} is limited to the nonsingular Gaussian reciprocal sequence. The nonsingularity assumption is critical for that model, because its well-posedness (i.e., the uniqueness of the sequence obeying the model) is guaranteed by the nonsingularity of the whole sequence. There is not any model for the general (singular/nonsingular) Gaussian reciprocal sequence in the literature, and it is not clear how to obtain such a model from the viewpoint of the literature on reciprocal sequences. For example, it is not clear how the model of \cite{Levy_Dynamic} can be extended to the general (singular/nonsingular) case. The CM viewpoint is very fruitful for studying reciprocal sequences. From it, the following theorem presents two models for the general (singular/nonsingular) Gaussian reciprocal sequence from the CM viewpoint. They are called reciprocal $CM_c$ models.

\begin{theorem}\label{Reciprocal_Model}
A ZMG $[x_k]$ is reciprocal iff it obeys $\eqref{CMc_Dynamic_Forward}$--$\eqref{CMc_Forward_BC2}$ and
\begin{align}
P_{k,i}=\left[ \begin{array}{cc}
P_{k,j} & P_{k,l}
\end{array}\right] \left[\begin{array}{cc}
P_{j} & P_{j,l}\\
P_{l,j} & P_{l}
\end{array}\right]^{+} \left[ \begin{array}{c}
P_{j,i}\\
P_{l,i}
\end{array}\right]\label{Reciprocal_P}
\end{align}
(i) for $c=N$ and $\forall i,j,k,l \in [0,N]$, $l <i < j < k $, and $\mathcal{G}$ given by $\eqref{L_2}$, or equivalently (ii) for $c=0$ and $\forall i,j,k,l \in [0,N]$, $ i < j < k < l $, and $\mathcal{G}$ given by $\eqref{F}$, where $P=(\mathcal{G})^{-1}G(\mathcal{G}')^{-1}$ and $G=\text{diag}(G_0,\ldots,G_N)$.

\end{theorem}
\begin{proof}
Every reciprocal sequence is $CM_c$. A ZMG sequence is $CM_c$ iff it obeys $\eqref{CMc_Dynamic_Forward}$--$\eqref{CMc_Forward_BC2}$. The covariance matrix of a sequence modeled by a $CM_c$ model can be calculated in terms of the parameters of the model and its boundary condition (the calculated covariance matrix is denoted by $P$ above). A Gaussian sequence is reciprocal iff its covariance function satisfies $\eqref{Reciprocal_1}$. Since model $\eqref{CMc_Dynamic_Forward}$--$\eqref{CMc_Forward_BC2}$ is for a $CM_c$ sequence, $P$ satisfies condition (b) of Theorem \ref{Reciprocal_Characterization_Theorem} (note that condition (b) of Theorem \ref{Reciprocal_Characterization_Theorem} is a $CM_c$ characterization for $c=N$ or $c=0$). So, a Gaussian sequence is reciprocal iff it obeys $\eqref{CMc_Dynamic_Forward}$--$\eqref{CMc_Forward_BC2}$ (for $c=N$ or $c=0$) and $P$ satisfies $\eqref{Reciprocal_P}$. 
\end{proof}

The results of this section support the idea of studying reciprocal sequences from the CM viewpoint.

\section{Characterizations and Dynamic Models of Other CM Sequences}\label{Section_Other_CM}

Sequences belonging to more than one class of CM sequences are useful for both application and theory. For example, an application of $CM_L \cap [0,k_2]$-$CM_L$ sequences is in trajectory modeling with a waypoint and a destination. Assume the destination density (at time $N$) as well as the density of the state at a time $k_2<N$ (i.e., waypoint information) is known. First, consider the waypoint information at $k_2$ (without destination information). In other words, we know the state density at $k_2$ but not after. With a CM evolution law between $0$ and $k_2$, such trajectories can be modeled by a $[0,k_2]$-$CM_L$ sequence. Now, consider the destination information (density) without the waypoint information. Such trajectories can be modeled by a $CM_L$ sequence \cite{DD_Conf}. Then, trajectories with a waypoint and a destination information can be modeled as a sequence being both $[0,k_2]$-$CM_L$ and $CM_L$, i.e., $CM_L \cap [0,k_2]$-$CM_L$. Sequences belonging to more than one CM class are also useful for theoretical purposes. For example, by Theorem \ref{CM_iff_Reciprocal}, a reciprocal sequence belongs to several CM classes. This is particularly useful for studying reciprocal sequences from the CM viewpoint (e.g., Theorem \ref{Reciprocal_Characterization_Theorem} and Theorem \ref{Reciprocal_Model}). Also, a dynamic model of $CM_L \cap [k_1,N]$-$CM_F$ sequences is useful for obtaining a full spectrum of models ranging from a $CM_L$ model to a reciprocal $CM_L$ model.

\begin{corollary}\label{CML_Dynamic_Forward_Proposition_I}
A Gaussian $[x_k]$ is $CM_L \cap [k_1,N]$-$CM_F$ iff it obeys $\eqref{CMc_Dynamic_Forward}$--$\eqref{CMc_Forward_BC2}$ (for $c=N$), and
\begin{align}
P_{k,i}&=\left[ \begin{array}{cc}
P_{k,j} & P_{k,k_1}
\end{array}\right] \left[\begin{array}{cc}
P_{j} & P_{j,k_1}\\
P_{k_1,j} & P_{k_1}
\end{array}\right]^{+} \left[ \begin{array}{c}
P_{j,i}\\
P_{k_1,i}
\end{array}\right]\label{k1k2_P}
\end{align}
$\forall i, j, k \in [k_1+1,N]$, $ i < j < k$, where
\begin{align}
P = \mathcal{G}^{-1}G(\mathcal{G}')^{-1}\label{P_Matrix_1}
\end{align}
$G=\text{diag}(G_0,\ldots,G_N)$, and $\mathcal{G}$ is given by $\eqref{L_2}$.

\end{corollary}
\begin{proof}
A sequence is $CM_L \cap [k_1,N]$-$CM_F$ iff it is $CM_L$ and $[k_1,N]$-$CM_F$. By Theorem \ref{CMc_Dynamic_Forward_Theorem}, a Gaussian sequence is $CM_L$ iff it obeys $\eqref{CMc_Dynamic_Forward}$--$\eqref{CMc_Forward_BC2}$ (for $c=N$). Also, the covariance matrix of a $CM_L$ sequence can be calculated as $\eqref{P_Matrix_1}$. On the other hand, by Corollary \ref{k1k2CMc_Characterization_Corollary}, a Gaussian sequence is $[k_1,N]$-$CM_F$ iff its covariance function satisfies $\eqref{k1k2_CML}$ (let $k_2=N$ and $c=k_1$ in Corollary \ref{k1k2CMc_Characterization_Corollary}). Therefore, a Gaussian sequence is $CM_L \cap [k_1,N]$-$CM_F$ iff it obeys $\eqref{CMc_Dynamic_Forward}$--$\eqref{CMc_Forward_BC2}$ and $\eqref{k1k2_P}$ holds. 
\end{proof}

Following the idea of Corollary \ref{CML_Dynamic_Forward_Proposition_I}, one can obtain models of other CM sequences belonging to more than one CM class, e.g., $CM_L \cap [k_1,N]$-$CM_F \cap [l_1,N]$-$CM_F$. As a result, by Theorem \ref{CM_iff_Reciprocal}, Corollary \ref{CML_Dynamic_Forward_Proposition_I}, and Theorem \ref{Reciprocal_Model}, one can see a full spectrum of models ranging from a $CM_L$ model (Theorem \ref{CMc_Dynamic_Forward_Theorem}) to a reciprocal $CM_L$ model (Theorem \ref{Reciprocal_Model}).

Characterizations presented in Corollary \ref{CMc_Characterization_Markov} and Corollary \ref{Reciprocal_Charactetization_Markov_Proposition} are based on the Markov property. To complete those characterizations, we need a characterization of the Gaussian Markov sequence based on the covariance function. A characterization was presented in \cite{Doob} for the scalar-valued Gaussian Markov process, but its generalization to the vector-valued case is not trivial. The following corollary presents a characterization of the vector-valued general (singular/nonsingular) Gaussian Markov sequence. To our knowledge, there is no such a characterization in the literature.

\begin{corollary}\label{Markov_Characterization}
A Gaussian $[x_k]$ with covariance function $C_{l_1,l_2}$ is Markov iff $C_{k,i}=C_{k,j}C^{+}_{j}C_{j,i}$, $\forall i, j, k \in [0,N]$, $ i < j < k$. 

\end{corollary}
\begin{proof}
Our proof is parallel to that of Theorem \ref{CMc_Characterization_Theorem}. The main differences are as follows. For the proof of necessity, instead of $r(k,j)$ in $\eqref{T1}$, we need to define $r(k,j)=x_k - E[x_k|x_j]$. Also, instead of Lemma \ref{GaussianCMc_Definition_Expectation}, we should use Lemma \ref{GaussianMarkov_Definition_E}. For the proof of sufficiency, instead of $r(k,j)$ in $\eqref{T4}$, we need to define $r(k,j)=x_k - C_{k,j}C_j^{+}x_j$. 
\end{proof}

Inspired by \cite{ABRAHAM}, a representation of the zero-mean nonsingular Gaussian (ZMNG) $CM_c$ sequence as a sum of a ZMNG Markov sequence and an uncorrelated ZMNG vector was presented in \cite{CM_Part_II_B_Conf}. We now extend it to the ZMG $CM_c$ sequence in Proposition \ref{CML_Markov_z_Proposition}, which can be proved based on Theorem \ref{CMc_Dynamic_Forward_Theorem}. We omit the proof for lack of space.
 
\begin{proposition}\label{CML_Markov_z_Proposition}
A ZMG $[x_k]$ is $CM_c, c \in \lbrace 0,N \rbrace$, iff it can be represented as $x_k=y_k+ \Gamma _k x_c, k \in [0,N] \setminus \lbrace c \rbrace$, where $[y_k] \setminus \lbrace y_c \rbrace $ is a ZMG Markov sequence, $x_c$ is a ZMG vector uncorrelated with $[y_k] \setminus \lbrace y_c \rbrace $, and $\Gamma_k$ are some matrices. 

\end{proposition}

A corollary of Proposition \ref{CML_Markov_z_Proposition} is as follows.

\begin{corollary}\label{CML_Decomp}
An $(N+1)d \times (N+1)d$ matrix (with $(N+1)$ blocks in each row/column and each block with dimension $d \times d$) is the covariance matrix of a $d$-dimensional vector-valued Gaussian $CM_c$ sequence iff $C=B+\Gamma D\Gamma ' $, where $D$ is a $d\times d$ positive semi-definite matrix and (i) for $c=N$, $B=\left[\begin{array}{cc}
B_1 & 0\\
0 & 0
\end{array}\right]$, $\Gamma =\left[\begin{array}{c}
S\\
I
\end{array}\right]$,
(ii) for $c=0$, $B=\left[\begin{array}{cc}
0 & 0\\
0 & B_1
\end{array}\right]$, $\Gamma =\left[\begin{array}{c}
I\\
S
\end{array}\right]$, where $B_1$ is an $Nd \times Nd$ covariance matrix of a $d$-vector Gaussian Markov sequence, $S$ is an arbitrary $Nd \times d$ matrix, and $I$ is the $d\times d$ identity matrix.  

\end{corollary}

\section{Summary and Conclusions}\label{Section_Summary}

In this paper, (singular/nonsingular) Gaussian CM, including reciprocal and Markov, sequences have been studied, modeled, characterized, and an application has been discussed. The obtained models (characterizations) for different classes of Gaussian CM sequences are complete (i.e., necessary and sufficient) descriptions of the classes. 

The nonsingularity assumption for the reciprocal model presented in \cite{Levy_Dynamic} is critical for the well-posedness of that model, and it has not been clear how that model can be extended to the singular case even after so many years. From the CM viewpoint, however, we have obtained two models and a characterization for the (singular/nonsingular) Gaussian reciprocal sequence. The well-posedness of our models is always guaranteed without any assumption. This demonstrates the significance of studying reciprocal sequences from the CM viewpoint. A full spectrum of models (characterizations) ranging from a $CM_L$ model (characterization) to a reciprocal $CM_L$ model (characterization) has been also presented. 

We have unified singular and nonsingular Gaussian CM (including reciprocal) sequences and provided tools for their application.

Reciprocal sequences were studied from the CM viewpoint in \cite{CM_Part_II_A_Conf}--\cite{CM_Part_II_B_Conf}, \cite{CM_Part_II_B_Conference}--\cite{Thesis_Reza}. Equivalent CM dynamic models were studied in \cite{CM_Conf_Explicitly}--\cite{CM_Journal_Algebraically}.

\end{document}